\documentclass[12pt]{article}
\usepackage{latexsym}
\usepackage{amsmath}
\usepackage{amssymb}
\usepackage{amscd}
\usepackage{array}
\usepackage{comment}
\usepackage[all]{xy}
\usepackage{amsthm}
\usepackage{amsfonts}
\usepackage{enumerate}
\usepackage{hyperref}
\usepackage{leftindex}
\usepackage{stackrel}

\numberwithin{equation}{section}

\newtheorem{theorem}{Theorem}[]
\newtheorem{proposition}[theorem]{Proposition}
\newtheorem{lemma}[theorem]{Lemma}

\newtheorem*{hypothesis*}{Hypothesis}

\theoremstyle{definition}

\newcommand{\Z}{\mathbf Z}

\newcommand{\GL}{\mathrm{GL}}

\newcommand{\Aut}{\operatorname{Aut}}

\newcommand{\Id}{\operatorname{Id}}

\begin{document}

\begin{center}
\Large{\bf Left braces of size $p^2q^2$}
		
\vspace{0.4cm}
\Large{Teresa Crespo}

\vspace{0.2cm}
\normalsize{Departament de Matem\`atiques i Inform\`atica, Universitat de Barcelona, Gran Via de les Corts Catalanes 585, 08007, Barcelona (Spain), e-mail:teresa.crespo@ub.edu}

\end{center}

\begin{abstract}
 We consider relatively prime integer numbers $m$ and $n$  such that each solvable group of order $mn$ has a normal subgroup of order $m$. We prove that each brace of size $mn$ is a semidirect product of a brace of size $m$ and a brace of size $n$. We further give a method to classify braces of size $mn$ from the classification of braces of sizes $m$ and $n$. We apply this result to determine all braces of size $p^2q^2$, for $p$ and $q$ odd primes satisfying some conditions which hold in particular for $p$ a Germain prime and $q=2p+1$.


\end{abstract}

\section{Introduction}

In \cite{R} Rump introduced braces to study
set-theoretic solutions of the Yang-Baxter equation. A (left) brace is a triple $(B,+,\cdot)$ where $B$ is a set and $+$ and $\cdot$ are binary operations such
that $(B, +)$ is an abelian group, $(B, \cdot)$ is a group and
$$a\cdot(b+c)+a = a\cdot b+a\cdot c,$$
for all $a, b, c \in B$. We call $(B, +$) the additive group and $(B, \cdot)$ the multiplicative group of the left brace. The cardinal of $B$ is called the size of the brace. If $(B,+)$ is an abelian group, then $(B,+,+)$ is a brace, called trivial brace.

Let $B_1$ and $B_2$ be left braces. A map $f : B_1 \to B_2$ is said to be a brace morphism if $f(b+b') = f(b)+f(b')$ and $f(b\cdot b') = f(b)\cdot f(b')$ for all $b,b' \in B_1$. If $f$ is bijective, we say that $f$ is an isomorphism. In that case we say that the braces $B_1$ and $B_2$ are isomorphic.

We recall the definition of direct and semidirect product of braces as defined in \cite{Ce} and \cite{SV}.
Let $(B_1,+,\cdot)$ and $(B_2,+,\cdot)$ be braces and $\tau:(B_2,\cdot)\to\Aut(B_1,+,\cdot)$
be a group morphism.
Define in $B_1\times B_2$ operations $+$ and $\cdot$ by

$$
(a,b)+(a',b')=(a+a',b+b'), \quad (a,b)\cdot(a',b')=(a\cdot \tau(b)(a'),b\cdot b').
$$

\noindent
Then $(B_1\times B_2,+,\cdot)$ is a brace which is called the semidirect product of the braces $B_1$ and $B_2$ via $\tau$ and will be denoted $B_1\rtimes_{\tau} B_2$.
If $\tau$ is the trivial morphism, then $(B_1\times B_2,+,\cdot)$ is called the direct product of $B_1$ and $B_2$.

We recall that, for a left brace $(B,+,\cdot)$ and each $a \in B$, we have a bijective map
$\lambda_a: B \rightarrow B$ defined by $\lambda_a(b)=-a+ a\cdot b$ which satisfies $\lambda_a(b+c)=\lambda_a(b)+\lambda_a(c), a\cdot b=a +\lambda_a(b),  \lambda_{a\cdot b}=\lambda_a \circ \lambda_b$, for any $a,b,c$ in $B$.

Left braces have been classified for sizes $p^2, p^3$, for $p$ a prime number (\cite{Ba2}); $pq$ and $p^2q$, for $p$ and $q$ odd prime numbers (\cite{AB, AB2, CCD, D}); $2p^2$, for $p$ an odd prime number (\cite{C}); $8p$, for $p$ an odd prime number $\neq 3, 7$ (\cite{CGRV1}) and for $12p$, for $p$ an odd prime number $\geq 7$ (\cite{CGRV2}). In this paper we consider relatively prime integer numbers $m$ and $n$  such that each solvable group of order $mn$ has a normal subgroup of order $m$. We prove that each brace of size $mn$ is a semidirect product of a brace of size $m$ and a brace of size $n$. We further give a method to classify braces of size $mn$ from the classification of braces of sizes $m$ and $n$. This is a generalization of the result obtained in \cite{CGRV2} in the case in which $m$ is prime. We apply our result to describe all braces of size $p^2q^2$, for $p$ and $q$ odd primes satisfying $q>p, q\geq 5, p \mid q-1, p\nmid q+1, p^2 \nmid q-1$. We note that these conditions hold in particular when $p$ is an odd Germain prime and $q=2p+1$.

\section{Left braces of size $mn$, for $\gcd(m,n)=1$}\label{method}

In this section we consider relatively prime integer numbers $m$ and $n$ and assume that each solvable group of order $mn$ has a normal subgroup of order $m$. We prove that each brace of order $mn$ is a semidirect product $B_1\rtimes_{\tau} B_2$, where $B_1$ is a brace of size $m$, $B_2$ is a brace of size $n$ and $\tau:(B_2,\cdot)\to\Aut(B_1,+,\cdot)$ is a group morphism. Moreover, given such $B_1$ and $B_2$,  we determine when two group morphisms $\sigma, \tau:(B_2,\cdot)\to\Aut(B_1,+,\cdot)$ provide isomorphic braces.

\begin{theorem}\label{str}
Let $m$ and $n$ be relatively prime integer numbers such that each solvable group of order $mn$ has a normal subgroup of order $m$. Then each brace of size $mn$ is a semidirect product of a brace of size $m$ and a brace of size $n$.
\end{theorem}

\begin{proof}
Let $(B,+,\cdot)$ be a brace of size $mn$. Let $B_1$ and $B_2$ be its unique additive subgroups of size $m$ and $n$, respectively. In particular $B_1$ and $B_2$ are characteristic subgroups in $(B,+)$. Since, for each $a \in B$, $\lambda_a$ is an automorphism of $(B,+)$, it leaves $B_1$ and $B_2$ setwise invariant. This implies that, for $a,b \in B_1$, we have $ab=a+\lambda_a(b) \in B_1$, as $\lambda_a(b) \in B_1$. Similarly, this can be applied to $B_2$. So, $B_1$ and $B_2$ are subbraces of $B$ and $B_1$ and $B_2$ are complements of one another. Let $a \in B_1$ and $b \in B_2$, then

$$ba=\leftindex^b{a}b \Rightarrow b+\lambda_b(a)=\leftindex^b{a}+\lambda_{^b a}(b).$$

\noindent
Since the multiplicative group of a brace is always solvable (see \cite{Ce} Theorem 5.2), our hypothesis implies that $(B_1,\cdot)$ is a normal subgroup of $(B,\cdot)$, hence $\leftindex^b{a} \in B_1$. Using again that the $\lambda$-action leaves $B_2$ setwise invariant, we obtain $\lambda_{^b a}(b) \in B_2$. A comparison of the components shows $\leftindex^b{a}=\lambda_b(a)$, i.e. under the $\lambda$-action, $(B_2,\cdot)$ acts by automorphisms of $(B_1,+)$ and $(B_1,\cdot)$, that is, by brace automorphisms. Analogously

$$ab=ba^b \Rightarrow a+\lambda_a(b)=b+\lambda_b(a^b),$$

\noindent
where $\lambda_a(b) \in B_2, \lambda_b(a^b) \in B_1$. Comparing components, we obtain $\lambda_a(b)=b$. Therefore $ab=a+\lambda_a(b)=a+b$ for $a \in B_1, b \in B_2$. Also, $ba=\leftindex^b{a}+\lambda_{^b a}(b)=\leftindex^b{a}+b=\tau_b(a)+b$ for an action $\tau:B_2 \rightarrow \Aut(B_1)$.

Finally, for $a,a' \in B_1; b,b' \in B_2$, we have

$$\begin{array}{lll} (a+b)(a'+b') &=& ab(a'+b')=a(ba'-b+bb')=a(\tau_b(a')+bb')\\ &=& a\tau_b(a')-a+a(bb')=a\tau_b(a')+bb',
\end{array}$$

\noindent
where we have use the brace condition in the second and fourth equalities. Hence

$$B \to B_1\rtimes_{\tau} B_2 \, ; \, a+b \mapsto (a,b)$$

\noindent
is indeed a brace morphism.

\end{proof}

We want to see now when two semidirect products of braces $B_1$ and $B_2$ of coprime orders are isomorphic.

\begin{proposition}\label{prop} Let $B_1, B_2$ be braces with $\gcd(|B_1|,|B_2|)=1$. Consider semidirect products $B_{\sigma}:=B_1\rtimes_{\sigma} B_2, B_{\tau}:=B_1\rtimes_{\tau} B_2$, for morphisms $\sigma,\tau:(B_2,\cdot) \to \Aut(B_1,+,\cdot)$. An isomorphism $h:B_{\sigma} \to B_{\tau}$ is of the form $(h_1,h_2)$, where $h_i \in \Aut(B_i), i=1,2$, and $h_1$ and $h_2$ satisfy

$$\tau h_2=\leftindex^{h_1}{\sigma}.$$

\end{proposition}

\begin{proof} The coprimality of $|B_1|$ and $|B_2|$ implies that the $B_i$ are subbraces of $B_{\sigma}$ and $B_{\tau}$ and furthermore, $(B_1,+)$ (respectively $(B_2,+)$) is the only subgroup of order $m$ (respectively $n$) in $(B_{\sigma},+)$ and $(B_{\tau},+)$. Hence an isomorphism $h:B_{\sigma} \to B_{\tau}$ is of the form $(h_1,h_2)$, where $h_i \in \Aut(B_i), i=1,2$. For $a,a' \in B_1, b,b' \in B_2$, we have

$$h((a,b)\cdot (a',b'))=h(a\sigma(b)(a'),bb')=(h_1(a\sigma(b)(a')),h_2(bb'))$$

\noindent
and

$$\begin{array}{lll} h(a,b)\cdot h(a',b')&=&(h_1(a),h_2(b))\cdot (h_1(a'),h_2(b'))\\ &=&(h_1(a)\tau(h_2(b))(h_1(a')),h_2(b)h_2(b')).
\end{array}$$

\noindent
We obtain

$$h_1(\sigma(b)(a')=\tau(h_2(b))(h_1(a')).$$

\noindent
Replacing $a'$ by $h_1^{-1}(a')$ results in the equation

$$h_1(\sigma(b)(h_1^{-1}(a'))=\tau(h_2(b))(a').$$

\noindent
As $a'$ and $b$ are arbitrary, this implies

$$\tau h_2=\leftindex^{h_1}{\sigma}.$$

\end{proof}

\section{Braces of size $p^2$, for $p$ an odd prime number}\label{Bach}

In \cite{Ba2} Bachiller obtained the classification of braces of sizes $p^2$ and $p^3$, up to isomorphism, for $p$ a prime number. We recall it for braces $(B,+,\cdot)$ of size $p^2$, for $p$ odd. We note that in this case $(B,\cdot)$ is isomorphic to $(B,+)$. For each brace, we give the group of brace automorphisms and an explicit isomorphism from $(B,\cdot)$ to $(B,+)$.

\subsection{$(B,+) \simeq \Z/(p^2)$}\label{cyclic}

There are two braces, up to isomorphism, with additive group isomorphic to $\Z/(p^2)$, the trivial one and a brace with $\cdot$ defined by

$$x_1\cdot x_2 =x_1+x_2+px_1x_2.$$

\noindent
In both cases, $(B,\cdot) \simeq \Z/(p^2)$. In the trivial case, we have $\Aut B= \Aut (\Z/(p^2))\simeq (\Z/(p^2))^*$. In the nontrivial case, we have

$$\Aut B= \{ k \in (\Z/(p^2))^* \, : \, k \equiv 1 \pmod{p} \}$$

\noindent
and an isomorphism from $(B,\cdot)$ into $\Z/(p^2)$ is given by $n\mapsto n-pn(n-1)/2$.

\subsection{$(B,+) \simeq \Z/(p)\times \Z/(p)$}\label{elem}

We write the elements in $\Z/(p)\times \Z/(p)$ in vector form. There are two braces, up to isomorphism, with additive group isomorphic to $\Z/(p)\times \Z/(p)$, the trivial one and a brace with $\cdot$ defined by

$$\left( \begin{matrix} x_1\\ y_1 \end{matrix} \right) \cdot \left(\begin{matrix} x_2 \\ y_2 \end{matrix} \right)=\left(\begin{matrix} x_1+x_2+y_1y_2\\ y_1+y_2\end{matrix} \right).$$

\noindent
In both cases, $(B,\cdot) \simeq \Z/(p)\times \Z/(p)$. In the trivial case, we have $\Aut B= \Aut (\Z/(p)\times \Z/(p))\simeq \GL(2,p)$. In the nontrivial case, we have

$$\Aut B= \left\{ \left( \begin{array}{cc} d^2 & b \\ 0 & d \end{array} \right) \, : \, b \in \Z/(p), d \in (\Z/(p))^*\right\}$$

\noindent
and an isomorphism from $(B,\cdot)$ into $\Z/(p) \times \Z/(p)$ is given by

$$\left(\begin{matrix} x\\ y \end{matrix} \right) \mapsto \left(\begin{matrix} x-y(y-1)/2\\ y \end{matrix} \right).$$

\section{Groups of order  $p^2q^2$\label{groups}}

We assume now that $p$ and $q$ are primes satisfying $p>2$, $q>p$ and $q\geq 5$.
These hypotheses imply that a group $G$ of order $p^2q^2$ has a unique normal $q$-Sylow $S_q$ of order $q^2$. Indeed, the number $n_q$ of $q$-Sylow subgroups of $G$ satisfies $n_q \in \{1,p,p^2\}$ and $n_q \equiv 1 \pmod{q}$. Clearly $q\nmid p-1$ and $q\mid p^2-1$ implies $q\mid p-1$ or $q\mid p+1$ but, if $q>p$, the second condition holds only for $p=2$ and $q=3$. We obtain that a group of order $p^2q^2$ is the semidirect product of a normal subgroup $S_q$ of order $q^2$ and a subgroup $S_p$ of order $p^2$. It is then determined by a group $G_1$ of order $q^2$, a group $G_2$ of order $p^2$ and a morphism $\tau:G_2 \rightarrow \Aut(G_1)$. We note that triples $(G_1,G_2,\tau)$ and $(G_1',G_2',\tau')$ provide isomorphic groups of order $p^2q^2$ if and only if there exist isomorphisms $f:G_1 \rightarrow G_1', g:G_2 \rightarrow G_2'$ such that $\leftindex^f{\tau}=\tau'  g$. The groups of order $p^2q^2$ may then be described by determining the equivalence classes of morphisms $\tau:G_2 \rightarrow \Aut(G_1)$ under the relation

$$\tau \sim \tau' \Leftrightarrow \exists (f,g) \in \Aut G_1 \times \Aut G_2 \, : \, \leftindex^f{\tau}=\tau'  g.$$

Let us further assume that $p$ and $q$ satisfy $p \mid q-1, p\nmid q+1$ and $p^2 \nmid q-1$.
If $G_1 \simeq \Z/(q^2)$ then $\Aut G_1 \simeq (\Z/(q^2))^*\simeq \Z/q(q-1)$. The assumptions $p \mid q-1$ and $p^2 \nmid q-1$ imply that $\Aut G_1$ contains a unique subgroup of order $p$ but no subgroup of order $p^2$.
If $G_1 \simeq \Z/(q) \times \Z/(q)$, then $\Aut G_1 \simeq \GL(2,q)$ and $|\GL(2,q)|=(q+1)q(q-1)^2$. The assumptions $p \mid q-1, p\nmid q+1$ and $p^2 \nmid q-1$ imply that $\Aut G_1$ contains elements of order $p$ but no element of order $p^2$.

Since $\tau$ and $\leftindex^{f}{\tau}$, for $f \in \GL(2,q)$, give isomorphic groups of order $p^2q^2$, we need to determine the subgroups of order $p$ of $\GL(2,q)$, up to conjugation. This is done in the following lemma which is easy to prove.

\begin{lemma}
For $\lambda$ a fixed generator of the unique subgroup of order $p$ of $\Z/(q)^*$,  a system of representatives of the conjugation classes of subgroups of order $p$ of $\GL(2,q)$ is

\begin{equation}\label{mat}
\left\langle \left( \begin{array}{cc}1&0\\0&\lambda \end{array}\right) \right\rangle,\left\langle \left( \begin{array}{cc}\lambda&0\\0&\lambda \end{array}\right) \right\rangle, \left\langle \left( \begin{array}{cc}\lambda&0\\0&\lambda^{-1} \end{array}\right) \right\rangle,\left\langle \left( \begin{array}{cc}\lambda&0\\0&\lambda^k \end{array}\right) \right\rangle,
\end{equation}

\noindent
for $k$ running over a system of representatives of elements of $(\Z/(p))^*$, different from $1$ and $-1$, under the relation $k \sim \l$ if and only if $k\l \equiv 1 \pmod{p}$.

The number of subgroups of order $p$ of $\GL(2,q)$ up to conjugation is then $(p+3)/2$.

\end{lemma}

We may now describe the groups of order $p^2q^2$ for primes $p$ and $q$ satisfying the following conditions.

\begin{equation}\label{hyp}
 q>p, p>2, q\geq 5, p \mid q-1, p\nmid q+1, p^2 \nmid q-1.
\end{equation}

\begin{lemma} Let $p$ and $q$ satisfying \eqref{hyp}. Let $G$ be a group of order $p^2q^2$ and let us denote by $S_q$ the unique $q$-Sylow subgroup of $G$.

\begin{enumerate}[1)]
\item Assume $S_q\simeq \Z/(q^2)$ and let $\alpha$ denote a fixed generator of the unique subgroup of order $p$ of $(\Z/(q^2))^*$. In this case, $G$ is isomorphic to one of the following groups.
\begin{enumerate}[{1.}1)]
\item $\Z/(p^2q^2)$;
\item $\Z/(q^2) \rtimes \Z/(p^2)$ with product given by

$$(x_1,y_1)\cdot(x_2,y_2)=(x_1+\alpha^{y_1}x_2,y_1+y_2);$$
\item $\Z/(pq^2)\times \Z/(p)$;
\item $\Z/(q^2) \rtimes (\Z/(p)\times \Z(p))$ with product given by
$$\left(x_1,\left(\begin{smallmatrix}y_1\\z_1\end{smallmatrix}\right)\right)\cdot
\left(x_2,\left(\begin{smallmatrix}y_2\\z_2\end{smallmatrix}\right)\right)=
\left(x_1+\alpha^{y_1}x_2,\left(\begin{smallmatrix}y_1+y_2\\z_1+z_2\end{smallmatrix}\right)\right).$$
\end{enumerate}
\item Assume $S_q\simeq \Z/(q)\times \Z/(q)$ and let $\lambda$ denote a fixed generator of the unique subgroup of order $p$ of $(\Z/(q))^*$. In this case, $G$ is isomorphic to one of the following groups.
\begin{enumerate}[{2.}1)]
\item $\Z/(p^2q)\times \Z/(q)$;
\item one of the $(p+3)/2$ groups $(\Z/(q)\times \Z/(q)) \rtimes_{M} \Z/(p^2)$ with product given by

$$\left(\left( \begin{array}{c} x_1 \\y_1 \end{array} \right),z_1\right) \cdot \left(\left( \begin{array}{c} x_2 \\y_2 \end{array} \right),z_2\right)=\left(\left( \begin{array}{c} x_1 \\y_1 \end{array} \right)+M^{z_1} \left( \begin{array}{c} x_2 \\y_2 \end{array} \right),z_1+z_2\right),$$

\noindent
where $M$ denotes one of the matrices in $\eqref{mat}$.
\item $\Z/(pq)\times \Z/(pq)$;
\item one of the $(p+3)/2$ groups $(\Z/(q)\times \Z/(q))\rtimes_{M}(\Z/(p)\times \Z/(p))$, with product given by

$$\left(\left( \begin{smallmatrix} x_1 \\y_1 \end{smallmatrix} \right),\left( \begin{smallmatrix} z_1 \\t_1 \end{smallmatrix} \right)\right) \cdot \left(\left( \begin{smallmatrix} x_2 \\y_2 \end{smallmatrix} \right),\left( \begin{smallmatrix} z_2 \\t_2 \end{smallmatrix} \right)\right)=\left(\left( \begin{smallmatrix} x_1 \\y_1 \end{smallmatrix} \right)+M^{z_1} \left( \begin{smallmatrix} x_2 \\y_2 \end{smallmatrix} \right),\left( \begin{smallmatrix} z_1 +z_2\\t_1+t_2 \end{smallmatrix} \right)\right),$$

\noindent
where $M$ denotes one of the matrices in $\eqref{mat}$;
\item $(\Z/(q)\times \Z/(q))\rtimes_{\lambda} (\Z/(p)\times \Z/(p))$ with product given by

$$\left(\left( \begin{smallmatrix} x_1 \\y_1 \end{smallmatrix} \right),\left( \begin{smallmatrix}  z_1 \\t_1 \end{smallmatrix} \right)\right) \cdot \left(\left( \begin{smallmatrix}x_2 \\y_2 \end{smallmatrix} \right),\left( \begin{smallmatrix} z_2 \\t_2 \end{smallmatrix} \right)\right)=\left(\left( \begin{smallmatrix} x_1 +\lambda^{t_1}x_2 \\y_1+\lambda^{z_1+t_1}y_2 \end{smallmatrix} \right),\left( \begin{smallmatrix} z_1 +z_2\\t_1 +t_2\end{smallmatrix} \right)\right).$$

\end{enumerate}

\end{enumerate}

\end{lemma}

\numberwithin{equation}{subsection}

\section{Left braces of size $p^2q^2$}

 In this section we consider primes $p$ and $q$ satisfying the conditions in \eqref{hyp}.
At the beginning of Section \ref{groups}, we have seen that, under these assumptions, $m=q^2$ and $n=p^2$ satisfy the conditions in Theorem \ref{str}. Hence, every brace of size $p^2q^2$ is the semidirect product of a brace $B_1$ of size $q^2$ and a brace $B_2$ of size $p^2$. We use the description of braces of order $p^2$ recalled in Section \ref{Bach} and Proposition \ref{prop} to determine all braces of size $p^2q^2$, for $p$ and $q$ satisfying the conditions \eqref{hyp}. We note that, in particular, these conditions are satisfied when $p$ is an odd Germain prime and $q=2p+1$.

For the description of the multiplicative groups of the braces of size $p^2q^2$ given below we shall use the explicit isomorphism from $(B_2,\cdot)$ to $(B_2,+)$ given in Sections \ref{cyclic} and \ref{elem}, respectively. Using these isomorphisms, one may prove that the description of the action of $\Aut B_2$ on $(B_2,\cdot)$ looks the same as its action on $(B_2,+)$ (see \cite{D} Lemma 7).

\subsection{$(B_1,+)=\Z/(q^2)$ and $(B_2,+)=\Z/(p^2)$}

In this section we describe braces of size $p^2q^2$ whose additive law is given by

\begin{equation}\label{addcc}
(x_1,x_2)+ (y_1,y_2)=(x_1+ y_1,x_2+y_2),
\end{equation}

\noindent
for $x_1,y_1 \in B_1; x_2,y_2 \in B_2$.

\subsubsection{$B_1$ trivial brace}

In this case, $\Aut B_1= (\Z/(q^2))^*$. Since $\Aut B_1$ is abelian, $\leftindex^{h_1}{\tau}=\tau$, for every morphism $\tau$ from $(B_2,\cdot)$ to $\Aut B_1$.

The morphisms from $\Z/(p^2)$ to $\Aut B_1$ are $\tau_i$ defined by $1 \mapsto \alpha^i$, for $\alpha$ a fixed generator of the unique subgroup of order $p$ of $\Aut B_1$, $0\leq i \leq p-1$, where $i=0$ corresponds to the trivial morphism.

\noindent
{\bf If $B_2$ is trivial,} for $h_2 \in \Aut B_2$ defined by $h_2(1)=i$, with $p \nmid i$, we have $\tau_i=\tau_1 h_2$. We obtain then two braces, the first one is the direct product of $B_1$ and $B_2$, with multiplicative law given by

\begin{equation}\label{mulcc1}
(x_1,x_2)\cdot (y_1,y_2)=(x_1+ y_1,x_2+y_2),
\end{equation}

\noindent
and the second one has multiplicative law given by

\begin{equation}\label{mulcc2}
(x_1,x_2)\cdot (y_1,y_2)=(x_1+\alpha^{x_2} y_1,x_2+y_2),
\end{equation}

\noindent
for $x_1,y_1 \in B_1; x_2,y_2 \in B_2;  \alpha$ a fixed element of order $p$ of $(\Z/(q^2))^*$.

\noindent
{\bf If $B_2$ is nontrivial,}  $\Aut B_2= \{ k \in (\Z/(p^2))^* \, : \, k \equiv 1 \pmod{p} \}$ and, for the morphisms $\tau_i$ defined above we have  $\tau_i h_2=\tau_i$, for each $h_2 \in \Aut B_2$. We obtain $p$ braces, including the direct product one. Taking into account the isomorphism from $(B_2,\cdot)$ into $\Z/(p^2)$ given in Section \ref{cyclic} and that $\alpha$ has order $p$, their multiplicative laws are given by

\begin{equation}\label{mulcc3}
(x_1,x_2)\cdot (y_1,y_2)=(x_1+\alpha^{i x_2} y_1,x_2+y_2+px_2y_2),
\end{equation}

\noindent
for $x_1,y_1 \in B_1; x_2,y_2 \in B_2; i=0,\dots,p-1; \alpha$ a fixed element of order $p$ of $(\Z/(q^2))^*$.

\subsubsection{$B_1$ nontrivial brace}

In this case, $\Aut B_1=\{ k \in (\Z/(q^2))^*:k\equiv 1 \pmod{q} \} \simeq \Z/(q)$. Then the unique morphism $\tau$ from $(B_2,\cdot)\simeq \Z/(p^2)$ to $\Aut B_1$ is the trivial one. We obtain two braces which are direct products of $B_1$ and $B_2$, where $B_2$ is either trivial or nontrivial. Their multiplicative laws are given by

\begin{equation}\label{mulcc4}
(x_1,x_2)\cdot (y_1,y_2)=(x_1+ y_1+qx_1y_1,x_2+y_2),
\end{equation}

\begin{equation}\label{mulcc5}
(x_1,x_2)\cdot (y_1,y_2)=(x_1+ y_1+qx_1y_1,x_2+y_2+px_2y_2),
\end{equation}

\noindent
for $x_1,y_1 \in B_1; x_2,y_2 \in B_2$.

\vspace{0.3cm}
Summing up, we have obtained the following result.

\begin{theorem} Let $p$ and $q$ be primes satisfying $q>p, q\geq 5, p \mid q-1, p\nmid q+1$ and $p^2 \nmid q-1$. There are $p+4$ braces with additive group $\Z/(p^2q^2)$. Four of them have multiplicative group $\Z/(p^2q^2)$ and the remaining $p$ have multiplicative group $\Z/(q^2) \rtimes \Z/(p^2)$.
\end{theorem}

\subsection{$(B_1,+)=\Z/(q^2)$ and $(B_2,+)=\Z/(p)\times \Z/(p)$}

In this section we describe braces of size $p^2q^2$ whose additive law is given by

\begin{equation}\label{addcn}
\left( x_1,\left( \begin{smallmatrix} y_1\\ z_1 \end{smallmatrix} \right)\right)+\left( x_2,\left( \begin{smallmatrix} y_2\\ z_2 \end{smallmatrix} \right)\right)=\left( x_1+x_2,\left( \begin{smallmatrix} y_1+y_2\\ z_1+z_2 \end{smallmatrix} \right)\right),
\end{equation}

\noindent
for $x_1, x_2 \in B_1; \left( \begin{smallmatrix} y_1\\ z_1 \end{smallmatrix} \right), \left( \begin{smallmatrix} y_2\\ z_2 \end{smallmatrix} \right)\in B_2$.

\subsubsection{$B_1$ trivial brace}\label{three}

In this case, $\Aut B_1\simeq (\Z/(q^2))^*$. Since $\Aut B_1$ is abelian, $\leftindex^{h_1} \tau=\tau$, for every morphism $\tau$ from $G_2$ to $\Aut B_1$ and $h_1 \in \Aut B_1$.

\noindent
{\bf If $B_2$ is trivial,} every nontrivial morphism $\tau:\Z/(p)\times \Z/(p) \rightarrow (\Z/(q^2))^*$ is equal to $\tau_0 h_2$, for $h_2 \in \Aut B_2\simeq \GL(2,p)$ and $\tau_0$ defined by $\tau_0\left( \begin{smallmatrix} 1\\ 0 \end{smallmatrix} \right)=\alpha$,$\tau_0\left( \begin{smallmatrix} 0\\ 1 \end{smallmatrix} \right)=1$, for $\alpha$ a fixed element of order $p$ in $(\Z/(q^2))^*$. We obtain one brace whose multiplicative law is given by

\begin{equation}\label{mulcn1}
\left( x_1,\left( \begin{smallmatrix} y_1\\ z_1 \end{smallmatrix} \right)\right)\cdot\left( x_2,\left( \begin{smallmatrix} y_2\\ z_2 \end{smallmatrix} \right)\right)=\left( x_1+\alpha^{y_1}x_2,\left( \begin{smallmatrix} y_1+y_2\\ z_1+z_2 \end{smallmatrix} \right)\right),
\end{equation}

\noindent
where $\alpha$ is  an element of order $p$ in $\Aut B_1$. Besides, we have the direct product of $B_1$ and $B_2$ with multiplicative law given by

\begin{equation}\label{mulcn2}
\left( x_1,\left( \begin{smallmatrix} y_1\\ z_1 \end{smallmatrix} \right)\right)\cdot \left( x_2,\left( \begin{smallmatrix} y_2\\ z_2 \end{smallmatrix} \right)\right)=\left( x_1+x_2,\left( \begin{smallmatrix} y_1+y_2\\ z_1+z_2 \end{smallmatrix} \right)\right),
\end{equation}

\noindent
{\bf If $B_2$ is nontrivial,} $\Aut B_2=\left\{\left( \begin{smallmatrix} d^2 & b\\ 0&d \end{smallmatrix} \right)\, : \, d \in (\Z/(p))^*, b \in \Z/(p)\right\}$. Every nontrivial morphism $\tau$ from $\Z/(p)\times \Z/(p)$ to $\Aut B_1$ is equal to $\tau_0 g$, for $g \in \GL(2,p)$ and $\tau_0$ defined by $\tau_0\left( \begin{smallmatrix} 1\\ 0 \end{smallmatrix} \right)=\alpha, \tau_0\left( \begin{smallmatrix} 0\\ 1 \end{smallmatrix} \right)=1$, for $\alpha$ a fixed element of order $p$ in $\Aut B_1$. By computation, we obtain that, for $g_1, g_2 \in \GL(2,p)$, we have $\tau_0 g_1=\tau_0 g_2$ if and only if the first rows of $g_1$ and $g_2$ are equal. We obtain then that the set of nontrivial morphisms $\tau$ from $\Z/(p)\times \Z/(p)$ to $\Aut B_1$ is precisely $\{ \tau_0 \left( \begin{smallmatrix} a & b\\ 0&1\end{smallmatrix} \right)\, : \, a \in (\Z/(p))^*, b \in \Z/(p) \}\cup \{ \tau_0 \left( \begin{smallmatrix} 0 & b\\ 1&0\end{smallmatrix} \right)\, : \, b \in (\Z/(p))^*\}$. Now, for $\tau:=\tau_0 \left( \begin{smallmatrix} a & b\\ 0&1\end{smallmatrix} \right), \tau':=\tau_0 \left( \begin{smallmatrix} a' & b'\\ 0&1\end{smallmatrix} \right)$, there exists $h_2 \in \Aut B_2$ such that $\tau' h_2=\tau$ if and only if $a'/a$ is a square; for $\tau:=\tau_0 \left( \begin{smallmatrix} 0 & b\\ 1&0\end{smallmatrix} \right), \tau':=\tau_0 \left( \begin{smallmatrix} 0 & b'\\ 1&0\end{smallmatrix} \right)$, there always exists $h_2 \in \Aut B_2$ such that $\tau' h_2=\tau$; for $\tau:=\tau_0 \left( \begin{smallmatrix} a & b\\ 0&1\end{smallmatrix} \right),  \tau':=\tau_0 \left( \begin{smallmatrix} 0 & b'\\ 1&0\end{smallmatrix} \right)$, there exists no $h_2 \in \Aut B_2$ such that $\tau' h_2=\tau$. We obtain then three braces. By considering the isomorphism from $(B_2,\cdot)$ into $\Z/(p)\times \Z/(p)$ given in Section \ref{elem}, their multiplicative laws are given by

\begin{equation}\label{mulcn3}
\left( x_1,\left( \begin{smallmatrix} y_1\\ z_1 \end{smallmatrix} \right)\right)\cdot\left( x_2,\left( \begin{smallmatrix} y_2\\ z_2 \end{smallmatrix} \right)\right)=\left( x_1+\alpha^{y_1-z_1(z_1-1)/2}x_2,\left( \begin{smallmatrix} y_1+y_2+z_1z_2\\ z_1+z_2 \end{smallmatrix} \right)\right),
\end{equation}

\begin{equation}\label{mulcn4}
\left( x_1,\left( \begin{smallmatrix} y_1\\ z_1 \end{smallmatrix} \right)\right)\cdot\left( x_2,\left( \begin{smallmatrix} y_2\\ z_2 \end{smallmatrix} \right)\right)=\left( x_1+\alpha^{a(y_1-z_1(z_1-1)/2)}x_2,\left( \begin{smallmatrix} y_1+y_2+z_1z_2\\ z_1+z_2 \end{smallmatrix} \right)\right),
\end{equation}

\noindent
and

\begin{equation}\label{mulcn5}
\left( x_1,\left( \begin{smallmatrix} y_1\\ z_1 \end{smallmatrix} \right)\right)\cdot\left( x_2,\left( \begin{smallmatrix} y_2\\ z_2 \end{smallmatrix} \right)\right)=\left( x_1+\alpha^{z_1}x_2,\left( \begin{smallmatrix} y_1+y_2+z_1z_2\\ z_1+z_2 \end{smallmatrix} \right)\right),
\end{equation}

\noindent
respectively, where $\alpha$ is  a fixed element of order $p$ in $\Aut B_1$ and $a$ is a fixed quadratic nonresidue modulo $p$. Besides, we have the direct product of $B_1$ and $B_2$ with multiplicative law given by

\begin{equation}\label{mulcn6}
\left( x_1,\left( \begin{smallmatrix} y_1\\ z_1 \end{smallmatrix} \right)\right)\cdot\left( x_2,\left( \begin{smallmatrix} y_2\\ z_2 \end{smallmatrix} \right)\right)=\left( x_1+x_2,\left( \begin{smallmatrix} y_1+y_2+z_1z_2\\ z_1+z_2 \end{smallmatrix} \right)\right),
\end{equation}

\subsubsection{$B_1$ nontrivial brace}

In this case, $\Aut B_1=\{ k \in (\Z/(q^2))^*:k\equiv 1 \pmod{q} \} \simeq \Z/(q)$. Then the unique morphism $\tau$ from $G_2\simeq \Z/(p)\times \Z/(p)$ to $\Aut B_1$ is the trivial one. We obtain then just two braces which are the direct product of $B_1$ and $B_2$, corresponding to $B_2$ trivial and $B_2$ nontrivial. Their multiplicative laws are given by

\begin{equation}\label{mulcn7}
\left( x_1,\left( \begin{smallmatrix} y_1\\ z_1 \end{smallmatrix} \right)\right)\cdot\left( x_2,\left( \begin{smallmatrix} y_2\\ z_2 \end{smallmatrix} \right)\right)=\left( x_1+x_2+qx_1x_2,\left( \begin{smallmatrix} y_1+y_2\\ z_1+z_2 \end{smallmatrix} \right)\right),
\end{equation}

\begin{equation}\label{mulcn8}
\left( x_1,\left( \begin{smallmatrix} y_1\\ z_1 \end{smallmatrix} \right)\right)\cdot\left( x_2,\left( \begin{smallmatrix} y_2\\ z_2 \end{smallmatrix} \right)\right)=\left( x_1+x_2+qx_1x_2,\left( \begin{smallmatrix} y_1+y_2+z_1z_2\\ z_1+z_2 \end{smallmatrix} \right)\right),
\end{equation}

\vspace{0.3cm}
Summing up, we have obtained the following result.

\begin{theorem} Let $p$ and $q$ be primes satisfying $q>p, q\geq 5, p \mid q-1, p\nmid q+1$ and $p^2 \nmid q-1$. There are eight braces with additive group $\Z/(pq^2)\times \Z/(p)$. Four of them have multiplicative group $\Z/(pq^2)\times \Z/(p)$ and the remaining four have multiplicative group $\Z/(q^2)\rtimes (\Z/(p)\times \Z/(p))$.
\end{theorem}

\subsection{$(B_1,+)=\Z/(q)\times \Z/(q)$ and $(B_2,+)=\Z/(p^2)$}

In this section we describe braces of size $p^2q^2$ whose additive law is given by

\begin{equation}\label{addnc}
\left( \left( \begin{smallmatrix} x_1\\ y_1 \end{smallmatrix} \right),z_1\right)+\left( \left( \begin{smallmatrix} x_2\\ y_2 \end{smallmatrix} \right),z_2\right)=\left(\left( \begin{smallmatrix} x_1+x_2\\ y_1+y_2 \end{smallmatrix} \right),z_1+z_2\right),
\end{equation}

\noindent
for $\left( \begin{smallmatrix} x_1\\ y_1 \end{smallmatrix} \right), \left( \begin{smallmatrix} x_2\\ y_2 \end{smallmatrix} \right)\in B_1, z_1,z_2 \in B_2$.

\subsubsection{$B_1$ trivial brace}

In this case, $\Aut B_1= \GL(2,q)$. Every morphism from $\Z/(p^2)$ to $\Aut B_1=\GL(2,q)$ is equal to $\leftindex^{h_1}{\tau}$ for some $h_1 \in \Aut B_1$ and $\tau$ defined by $\tau(1)=M^{\ell}$ for $M$ one of the matrices in \eqref{mat} and $1\leq \ell \leq p-1$.

\noindent
{\bf If $B_2$ is trivial,} $\Aut B_2=\Aut \Z/(p^2)$. For $\tau:\Z/(p^2) \rightarrow \Aut B_1$ defined by $\tau(1)=M$ and $h_2 \in \Aut \Z/(p^2)$, we have $\tau h_2(1)=M^{h_2(1)}$. Hence for morphisms $\tau, \tau'$ with $\tau(1)=M$ and $\tau'(1)=M^{\ell}$, one has $\tau \sim \tau'$. We have then one brace for each conjugation class of subgroups of order $p$ in $\GL(2,q)$. We obtain $(p+3)/2$ braces, whose multiplicative laws are given by

\begin{equation}\label{mulnc1}
\left( \left( \begin{smallmatrix} x_1\\ y_1 \end{smallmatrix} \right),z_1\right)\cdot\left( \left( \begin{smallmatrix} x_2\\ y_2 \end{smallmatrix} \right),z_2\right)=\left(\left( \begin{smallmatrix} x_1\\ y_1 \end{smallmatrix} \right)+M^{z_1}\left( \begin{smallmatrix} x_2\\ y_2 \end{smallmatrix} \right),z_1+z_2\right),
\end{equation}

\noindent
for $M$ one of the matrices in \eqref{mat}. Besides, we obtain the direct product of $B_1$ and $B_2$ whose multiplicative law is given by

\begin{equation}\label{mulnc2}
\left( \left( \begin{smallmatrix} x_1\\ y_1 \end{smallmatrix} \right),z_1\right)\cdot\left( \left( \begin{smallmatrix} x_2\\ y_2 \end{smallmatrix} \right),z_2\right)=\left(\left( \begin{smallmatrix} x_1+x_2\\ y_1+y_2 \end{smallmatrix} \right),z_1+z_2\right),
\end{equation}

\noindent
{\bf If $B_2$ is nontrivial,} we have $\Aut B_2=\{ k \in (\Z/(p^2))^* \,:\, k\equiv 1 \pmod{p} \}$. Since a nontrivial morphism $\tau$ from $(B_2,\cdot)$ to $\Aut B_1$ sends 1 to an element of order $p$, we have $\tau h_2=\tau$ for $h_2 \in \Aut B_2$. As noted above, a nontrivial morphism $\tau$ from $\Z/(p^2)$ to $\Aut B_1$ is equal to $\leftindex^{h_1}{\tau}$ for some $h_1 \in \Aut B_1$ and $\tau$ defined by $\tau(1)=M^{\ell}$ for $M$ one of the matrices in \eqref{mat} and $1\leq \ell \leq p-1$. Let us see if for some $\ell\in \{2,\dots,p-1\}$ and some matrix $M$ in \eqref{mat}, the matrices $M$ and $M^{\ell}$ are conjugate by some element in $\GL(2,q)$. This is so only for $M=\left(\begin{smallmatrix} \lambda &0\\0&\lambda^{-1} \end{smallmatrix} \right)$ and $\ell=p-1$. In this case, there are $p-1$ braces for each matrix $M$ different from $\left(\begin{smallmatrix} \lambda &0\\0&\lambda^{-1} \end{smallmatrix} \right)$ and $(p-1)/2$ for this last one. By considering the isomorphism from $(B_2,\cdot)$ into $\Z/(p^2)$ given in Section \ref{cyclic} and taking into account that $M$ denotes a matrix of order $p$, we obtain $\dfrac {p+1}{2} (p-1)+\dfrac{p-1} 2=\dfrac{(p-1)(p+2)} 2$ braces whose multiplicative laws are given by

\begin{equation}\label{mulnc3}
\left( \left( \begin{smallmatrix} x_1\\ y_1 \end{smallmatrix} \right),z_1\right)\cdot\left( \left( \begin{smallmatrix} x_2\\ y_2 \end{smallmatrix} \right),z_2\right)=\left(\left( \begin{smallmatrix} x_1\\ y_1 \end{smallmatrix} \right)+M^{\ell z_1}\left( \begin{smallmatrix} x_2\\ y_2 \end{smallmatrix} \right),z_1+z_2+pz_1z_2\right),
\end{equation}

\noindent
for $M$ one of the matrices in \eqref{mat} and with $1\leq \ell \leq p-1$, for $M\neq \left(\begin{smallmatrix} \lambda &0\\0&\lambda^{-1} \end{smallmatrix} \right)$; $1\leq \ell \leq (p-1)/2$, for $M=\left(\begin{smallmatrix} \lambda &0\\0&\lambda^{-1} \end{smallmatrix} \right)$.

Besides, we obtain the direct product of $B_1$ and $B_2$ whose multiplicative law is given by

\begin{equation}\label{mulnc4}
\left( \left( \begin{smallmatrix} x_1\\ y_1 \end{smallmatrix} \right),z_1\right)\cdot\left( \left( \begin{smallmatrix} x_2\\ y_2 \end{smallmatrix} \right),z_2\right)=\left(\left( \begin{smallmatrix} x_1+x_2\\ y_1+y_2 \end{smallmatrix} \right),z_1+z_2+pz_1z_2\right),
\end{equation}

\subsubsection{$B_1$ nontrivial brace}

If $B_1$ is nontrivial, $\Aut B_1=\{ \left(\begin{smallmatrix} d^2 & b \\ 0 & d \end{smallmatrix}\right) : b \in \Z/(q),d\in (\Z/(q^2))^* \}$. The matrices of order $p$ in $\Aut B_1$ are conjugate to some diagonal matrix of the form  $\left(\begin{smallmatrix} d^2 & 0 \\ 0 & d \end{smallmatrix}\right)$ with $d$ an element of order $p$ in $(\Z/(q))^*$. For $\lambda$ a chosen element of order $p$ in $(\Z/(q))^*$, the morphisms $\tau$ from $\Z/(p^2)$ to $\Aut B_1$ are given by $\tau(1)=\left(\begin{smallmatrix} \lambda^2 & 0 \\ 0 & \lambda \end{smallmatrix}\right)^{\ell}$, for $1\leq \ell \leq p-1$. We note that $\left(\begin{smallmatrix} \lambda^2 & 0 \\ 0 & \lambda \end{smallmatrix}\right)=\left(\begin{smallmatrix} \lambda & 0 \\ 0 & \lambda^k \end{smallmatrix}\right)^2$, with $k=(p+1)/2$.

\noindent
{\bf If $B_2$ is trivial,} for $\tau:\Z/(p^2)\rightarrow \Aut B_1$ defined by $\tau(1)=M$, we have $\tau h_2 (1)=M^{h_2(1)}$. Hence for morphisms $\tau, \tau'$ with $\tau(1)=M$ and $\tau'(1)=M^{\ell}$, one has $\tau \sim \tau'$. We may then reduce to the case where $\tau(1)=\left(\begin{smallmatrix} \lambda^2 & 0 \\ 0 & \lambda \end{smallmatrix}\right)$ and we obtain one brace whose multiplicative law is given by

\begin{equation}\label{mulnc5}
\left( \left( \begin{smallmatrix} x_1\\ y_1 \end{smallmatrix} \right),z_1\right)\cdot\left( \left( \begin{smallmatrix} x_2\\ y_2 \end{smallmatrix} \right),z_2\right)=\left(\left( \begin{smallmatrix} x_1+\lambda^{2z_1} x_2+\lambda^{2z_1} x_1x_2\\ y_1 +\lambda^{z_1} y_2\end{smallmatrix} \right),z_1+z_2\right).
\end{equation}

\noindent
Besides, we have the direct product whose multiplicative law is given by

\begin{equation}\label{mulnc6}
\left( \left( \begin{smallmatrix} x_1\\ y_1 \end{smallmatrix} \right),z_1\right)\cdot\left( \left( \begin{smallmatrix} x_2\\ y_2 \end{smallmatrix} \right),z_2\right)=\left(\left( \begin{smallmatrix} x_1+ x_2+x_1x_2\\ y_1 +y_2\end{smallmatrix} \right),z_1+z_2\right).
\end{equation}

\noindent
{\bf If $B_2$ is nontrivial,} we have $\Aut B_2=\{ k \in (\Z/(p^2))^* \,:\, k\equiv 1 \pmod{p} \}$, as above. For $h_2 \in \Aut B_2$ and $\tau:(B_2,\cdot) \rightarrow \Aut B_1$, we have $\tau h_2=\tau$. We obtain then $p-1$ braces.  By considering again the isomorphism from $(B_2,\cdot)$ into $\Z/(p^2)$ given in Section \ref{cyclic} and taking into account that $\tau(1)$ is a matrix of order $p$, their multiplicative laws are given by

\begin{equation}\label{mulnc7}
\left( \left( \begin{smallmatrix} x_1\\ y_1 \end{smallmatrix} \right),z_1\right)\cdot\left( \left( \begin{smallmatrix} x_2\\ y_2 \end{smallmatrix} \right),z_2\right)=\left(\left( \begin{smallmatrix} x_1+\lambda^{2\ell z_1} x_2+\lambda^{2\ell z_1}x_1x_2\\ y_1 +\lambda^{\ell z_1} y_2\end{smallmatrix} \right),z_1+z_2+pz_1z_2\right),
\end{equation}

\noindent
where $\lambda$ is a fixed element of order $p$ in $(\Z/(q))^*$ and $1\leq \ell \leq p-1$. Besides, we have the direct product whose multiplicative law is given by

\begin{equation}\label{mulnc8}
\left( \left( \begin{smallmatrix} x_1\\ y_1 \end{smallmatrix} \right),z_1\right)\cdot\left( \left( \begin{smallmatrix} x_2\\ y_2 \end{smallmatrix} \right),z_2\right)=\left(\left( \begin{smallmatrix} x_1+ x_2+x_1x_2\\ y_1 +y_2\end{smallmatrix} \right),z_1+z_2+pz_1z_2\right).
\end{equation}

\vspace{0.3cm}
Summing up, we have obtained the following result.

\begin{theorem} Let $p$ and $q$ be primes satisfying $q>p, q\geq 5, p \mid q-1, p\nmid q+1$ and $p^2 \nmid q-1$. There are $(p^2+4p+9)/2$ braces with additive group $\Z/(p^2q)\times \Z/(q)$.

\begin{enumerate}[a)]
\item There are four such braces with multiplicative group $\Z/(p^2q)\times \Z/(q)$;
\item for each of the matrices $M$ in \eqref{mat} different from $\left( \begin{smallmatrix} \lambda & 0 \\ 0& \lambda^{-1} \end{smallmatrix} \right)$ and $\left( \begin{smallmatrix} \lambda&0\\ 0&\lambda^{(p+1)/2} \end{smallmatrix} \right)$, there are $p$ such braces with multiplicative group $(\Z/(q)\times \Z/(q))\rtimes_M \Z/(p^2)$;
\item for $M=\left( \begin{smallmatrix} \lambda & 0 \\ 0& \lambda^{-1} \end{smallmatrix} \right)$, there are $(p+1)/2$ such braces with  multiplicative group $(\Z/(q)\times \Z/(q))\rtimes_M \Z/(p^2)$;
\item  for $M=\left( \begin{smallmatrix} \lambda&0\\ 0&\lambda^{(p+1)/2} \end{smallmatrix} \right)$, there are $2p$ such braces with multiplicative group $(\Z/(q)\times \Z/(q))\rtimes_M \Z/(p^2)$.
\end{enumerate}
\end{theorem}

\subsection{$(B_1,+)=\Z/(q)\times \Z/(q)$ and $(B_2,+)=\Z/(p)\times \Z/(p)$}

In this section we describe braces of size $p^2q^2$ whose additive law is given by

\begin{equation}\label{addnn}
\left( \left( \begin{smallmatrix} x_1\\ y_1 \end{smallmatrix} \right),\left( \begin{smallmatrix} z_1\\ t_1 \end{smallmatrix} \right)\right)+\left( \left( \begin{smallmatrix} x_2\\ y_2 \end{smallmatrix} \right),\left( \begin{smallmatrix} z_2\\ t_2 \end{smallmatrix} \right)\right)=\left( \left( \begin{smallmatrix} x_1+x_2\\ y_1+y_2 \end{smallmatrix} \right),\left( \begin{smallmatrix} z_1+z_2\\ t_1+t_2 \end{smallmatrix} \right)\right),
\end{equation}

\noindent
for $\left( \begin{smallmatrix} x_1\\ y_1 \end{smallmatrix} \right),\left( \begin{smallmatrix} x_2\\ y_2 \end{smallmatrix} \right) \in B_1; \left( \begin{smallmatrix} z_1\\ t_1 \end{smallmatrix} \right), \left( \begin{smallmatrix} z_2\\ t_2 \end{smallmatrix} \right)\in B_2$.

\subsubsection{$B_1$ trivial brace}

In this case, $\Aut B_1= \GL(2,q)$.  A nontrivial morphism $\tau$ from $\Z/(p)\times \Z/(p)$ to $\Aut B_1$ either has an order $p$ kernel or is injective. In the first case, it is equal to $\leftindex^{h_1}{\tau}$ for some $h_1 \in \Aut B_1$ and $\tau$ defined by $\tau(u)=M$, $\tau(v)=\Id$, for some $\Z/(p)$-basis $(u,v)$ of $\Z/(p)\times \Z/(p)$, where $M$ is one of the matrices in \eqref{mat}. In the second case, it is equal to $\leftindex^{h_1}{\tau}$ for some $h_1 \in \Aut B_1$ and $\tau$ defined by $\tau(u)=\left( \begin{smallmatrix}1&0\\0&\lambda \end{smallmatrix}\right), \tau(v)=\left( \begin{smallmatrix}\lambda&0\\0&\lambda \end{smallmatrix}\right)$, for an element $\lambda$ of order $p$ in $(\Z/(q))^*$ and some basis $(u,v)$ of $\Z/(p)\times \Z/(p)$. Indeed, all subgroups of order $p^2$ of $\GL(2,q)$ are conjugate, as they are the $p$-Sylow subgroups of $\GL(2,q)$, and $\left( \begin{smallmatrix}1&0\\0&\lambda \end{smallmatrix}\right)$ and $\left( \begin{smallmatrix}\lambda&0\\0&\lambda \end{smallmatrix}\right)$ are a basis of the subgroup of order $p^2$ whose elements are diagonal matrices.

\noindent
{\bf If $B_2$ is trivial,} we have $\Aut B_2=\GL(2,p)$. For $\tau$ defined by $\tau(u)=M$, $\tau(v)=\Id$, for some $\Z/(p)$-basis $(u,v)$ of $\Z/(p)\times \Z/(p)$, we have $\tau=\tau_0 h_2$, for $h_2$ defined by $h_2(u)=\left( \begin{smallmatrix} 1\\ 0 \end{smallmatrix} \right), h_2(v)=\left( \begin{smallmatrix} 0\\ 1 \end{smallmatrix} \right)$ and $\tau_0$ defined by $\tau_0\left( \begin{smallmatrix} 1\\ 0 \end{smallmatrix} \right)=M,\tau_0\left( \begin{smallmatrix} 0\\ 1 \end{smallmatrix} \right)=\Id$. We obtain then $(p+3)/2$ braces whose multiplicative laws are given by

\begin{equation}\label{mulnn1}
\left( \left( \begin{smallmatrix} x_1\\ y_1 \end{smallmatrix} \right),\left( \begin{smallmatrix} z_1\\ t_1 \end{smallmatrix} \right)\right)\cdot\left( \left( \begin{smallmatrix} x_2\\ y_2 \end{smallmatrix} \right),\left( \begin{smallmatrix} z_2\\ t_2 \end{smallmatrix} \right)\right)=\left(\left( \begin{smallmatrix} x_1\\ y_1 \end{smallmatrix} \right)+M^{z_1}\left( \begin{smallmatrix} x_2\\ y_2 \end{smallmatrix} \right),\left( \begin{smallmatrix} z_1+z_2\\ t_1+t_2 \end{smallmatrix} \right)\right),
\end{equation}

\noindent
for $M$ one of the matrices in \eqref{mat}. In the case when $\tau$ is injective, for an adequate $h_2$, we have $\tau=\tau_0 h_2$, for $\tau_0$ defined by $\tau_0\left(\begin{smallmatrix} 1\\0 \end{smallmatrix} \right)=\left(\begin{smallmatrix} 1&0\\0&\lambda \end{smallmatrix} \right),\tau_0\left(\begin{smallmatrix} 0\\1 \end{smallmatrix} \right)=\left(\begin{smallmatrix} \lambda&0\\0&\lambda \end{smallmatrix} \right)$, where $\lambda$ is a fixed element of order $p$ in $(\Z/(q))^*$. We obtain then one brace whose multiplicative law is given by

\begin{equation}\label{mulnn2}
\left( \left( \begin{smallmatrix} x_1\\ y_1 \end{smallmatrix} \right),\left( \begin{smallmatrix} z_1\\ t_1 \end{smallmatrix} \right)\right)\cdot\left( \left( \begin{smallmatrix} x_2\\ y_2 \end{smallmatrix} \right),\left( \begin{smallmatrix} z_2\\ t_2 \end{smallmatrix} \right)\right)=\left(\left( \begin{smallmatrix} x_1+\lambda^{t_1}x_2\\ y_1+\lambda^{z_1+t_1} y_2 \end{smallmatrix} \right),\left( \begin{smallmatrix} z_1+z_2\\ t_1+t_2 \end{smallmatrix} \right)\right),
\end{equation}

\noindent
for $\lambda$ a fixed element of order $p$ in $(\Z/(q))^*$. Besides, we have the direct product, whose multiplicative law is given by

\begin{equation}\label{mulnn3}
\left( \left( \begin{smallmatrix} x_1\\ y_1 \end{smallmatrix} \right),\left( \begin{smallmatrix} z_1\\ t_1 \end{smallmatrix} \right)\right)\cdot\left( \left( \begin{smallmatrix} x_2\\ y_2 \end{smallmatrix} \right),\left( \begin{smallmatrix} z_2\\ t_2 \end{smallmatrix} \right)\right)=\left(\left( \begin{smallmatrix} x_1+x_2\\ y_1 +y_2\end{smallmatrix} \right),\left( \begin{smallmatrix} z_1+z_2\\ t_1+t_2 \end{smallmatrix} \right)\right).
\end{equation}

\noindent
{\bf If $B_2$ is nontrivial,} we have $\Aut B_2=\left\{\left( \begin{smallmatrix} d^2 & b\\ 0&d \end{smallmatrix} \right)\, : \, d \in (\Z/(p))^*, b \in \Z/(p)\right\}$, as in Section \ref{three}. Now every morphism $\tau$ from $\Z/(p)\times \Z/(p)$ to $\Aut B_1$ with an order $p$ kernel is equal to $\tau_0 g$, for $g \in \GL(2,p)$ and $\tau_0$ defined by $\tau_0\left( \begin{smallmatrix} 1\\ 0 \end{smallmatrix} \right)=M, \tau_0\left( \begin{smallmatrix} 0\\ 1\end{smallmatrix} \right)=\Id$, for $M$ one of the matrices in \eqref{mat}. Similarly as in Section \ref{three}, we obtain that the set of nontrivial morphisms $\tau$ from $\Z/(p)\times \Z/(p)$ to $\Aut B_1$ is precisely $\{ \tau_0 \left( \begin{smallmatrix} a & b\\ 0&1\end{smallmatrix} \right)\, : \, a \in (\Z/(p))^*, b \in \Z/(p) \}\cup \{ \tau_0 \left( \begin{smallmatrix} 0 & b\\ 1&0\end{smallmatrix} \right)\, : \, b \in (\Z/(p))^*\}$. Moreover, again as in Section \ref{three}, under the relation

$$\tau \sim \tau' \Leftrightarrow \exists h_2 \in \Aut B_2 \, : \, \tau' h_2=\tau,$$

\noindent
we are left with $\tau_0, \tau_0\left( \begin{smallmatrix} a & 0\\ 0&1\end{smallmatrix} \right)$, for $a \in (\Z/(p))^*$ a non-square element, and $\tau_0 \left( \begin{smallmatrix} 0 & 1\\ 1&0\end{smallmatrix} \right)$.  Now, if $M=\left( \begin{smallmatrix} \lambda & 0\\ 0&\lambda^{-1}\end{smallmatrix} \right)$, the matrices $M$ and $M^{-1}$ are conjugate by $\left( \begin{smallmatrix} 0 & 1\\ 1&0\end{smallmatrix} \right) \in \GL(2,q)=\Aut B_1$. Hence, for $h_1=\left( \begin{smallmatrix} 0 & 1\\ 1&0\end{smallmatrix} \right)$, we have $\leftindex^{h_1}{\tau}_0= \tau_0 \left( \begin{smallmatrix} -1 & 0\\ 0&1\end{smallmatrix} \right)$ which implies that, if $-1$ is not a square in $\Z/(p)$, then the orbits corresponding to $\tau_0$ and $\tau_0\left( \begin{smallmatrix} a & 0\\ 0&1\end{smallmatrix} \right)$ coincide. We obtain then two braces corresponding to $M=\left( \begin{smallmatrix} \lambda & 0\\ 0&\lambda^{-1}\end{smallmatrix} \right)$ and three corresponding to the other matrices. Summing up, there are $(3/2)(p+3)$ braces if $p\equiv 1 \pmod{4}$ and  $(3/2)(p+3)-1 $ braces if $p\equiv 3 \pmod{4}$. Taking into account the isomorphism from $(B_2,\cdot)$ into $\Z/(p)\times \Z/(p)$ given in Section \ref{elem}, the corresponding multiplicative laws are given by

\begin{equation}\label{mulnn4}
\left( \left( \begin{smallmatrix} x_1\\ y_1 \end{smallmatrix} \right),\left( \begin{smallmatrix} z_1\\ t_1 \end{smallmatrix} \right)\right)\cdot\left( \left( \begin{smallmatrix} x_2\\ y_2 \end{smallmatrix} \right),\left( \begin{smallmatrix} z_2\\ t_2 \end{smallmatrix} \right)\right)=\left( \left( \begin{smallmatrix} x_1\\ y_1 \end{smallmatrix} \right)+M^{z_1-t_1(t_1-1)/2} \left( \begin{smallmatrix} x_2\\ y_2 \end{smallmatrix} \right),\left( \begin{smallmatrix} z_1+z_2+t_1t_2\\ t_1+t_2 \end{smallmatrix} \right)\right),
\end{equation}

\begin{equation}\label{mulnn5}
\left( \left( \begin{smallmatrix} x_1\\ y_1 \end{smallmatrix} \right),\left( \begin{smallmatrix} z_1\\ t_1 \end{smallmatrix} \right)\right)\cdot\left( \left( \begin{smallmatrix} x_2\\ y_2 \end{smallmatrix} \right),\left( \begin{smallmatrix} z_2\\ t_2 \end{smallmatrix} \right)\right)=\left( \left( \begin{smallmatrix} x_1\\ y_1 \end{smallmatrix} \right)+M^{a(z_1-t_1(t_1-1)/2)} \left( \begin{smallmatrix} x_2\\ y_2 \end{smallmatrix} \right),\left( \begin{smallmatrix} z_1+z_2+t_1t_2\\ t_1+t_2 \end{smallmatrix} \right)\right),
\end{equation}

\noindent
and

\begin{equation}\label{mulnn6}
\left( \left( \begin{smallmatrix} x_1\\ y_1 \end{smallmatrix} \right),\left( \begin{smallmatrix} z_1\\ t_1 \end{smallmatrix} \right)\right)\cdot\left( \left( \begin{smallmatrix} x_2\\ y_2 \end{smallmatrix} \right),\left( \begin{smallmatrix} z_2\\ t_2 \end{smallmatrix} \right)\right)=\left( \left( \begin{smallmatrix} x_1\\ y_1 \end{smallmatrix} \right)+M^{t_1} \left( \begin{smallmatrix} x_2\\ y_2 \end{smallmatrix} \right),\left( \begin{smallmatrix} z_1+z_2+t_1t_2\\ t_1+t_2 \end{smallmatrix} \right)\right),
\end{equation}

\noindent
respectively, where $M$ is one of the matrices in \eqref{mat} and $a$ is a fixed quadratic nonresidue modulo $p$ with the exception that, for $p\equiv 3 \pmod{4}$ and $M=\left( \begin{smallmatrix} \lambda & 0\\ 0&\lambda^{-1}\end{smallmatrix} \right)$, the braces with multiplicative laws  \eqref{mulnn4} and \eqref{mulnn5} are isomorphic.

As established above, an injective morphism $\tau$ from $\Z/(p)\times \Z/(p)$ to $\Aut B_1$ is equal to $\leftindex^{h_1}{\tau}$ for some $h_1 \in \GL(2,q)$ and $\tau$ defined by $\tau(u)=\left( \begin{smallmatrix}1&0\\0&\lambda \end{smallmatrix}\right), \tau(v)=\left( \begin{smallmatrix}\lambda&0\\0&\lambda \end{smallmatrix}\right)$, for an element $\lambda$ of order $p$ in $(\Z/(q))^*$ and some $\Z/(p)$-basis $(u,v)$ of $\Z/(p)\times \Z/(p)$. A transversal of $\Aut B_2$ in $\GL(2,p)$ is

$$\left\{ \left( \begin{smallmatrix}a&0\\c&1 \end{smallmatrix}\right)\, : \, a \in (\Z/(p))^*, c \in \Z/(p) \right\} \cup
\left\{ \left( \begin{smallmatrix}0&c\\1&0\end{smallmatrix}\right)\, : \, c \in (\Z/(p))^* \right\},$$

\noindent
hence any injective morphism $\tau$ from $\Z/(p)\times \Z/(p)$ to $\Aut B_1$ is equivalent under the relation in Proposition \ref{prop} either to $\tau_{a,c}=\tau_0 h_2$ for $h_2=\left( \begin{smallmatrix}a&0\\c&1 \end{smallmatrix}\right)$ for some $a \in (\Z/(p))^*, c \in \Z/(p)$ or to $\tau_{c}=\tau_0 h_2$ for $h_2= \left( \begin{smallmatrix}0&c\\1&0\end{smallmatrix}\right)$ for some $c \in (\Z/(p))^*$, where  $\tau_0$ is defined by $\tau_0\left( \begin{smallmatrix} 1\\ 0 \end{smallmatrix} \right)=\left( \begin{smallmatrix}1&0\\0&\lambda \end{smallmatrix}\right), \tau_0\left( \begin{smallmatrix} 0\\ 1 \end{smallmatrix} \right)=\left( \begin{smallmatrix}\lambda&0\\0&\lambda \end{smallmatrix}\right)$. Now the normalizer of $\left\langle \left( \begin{smallmatrix}1&0\\0&\lambda \end{smallmatrix}\right), \left( \begin{smallmatrix}\lambda&0\\0&\lambda \end{smallmatrix}\right)\right\rangle$ in $\GL(2,q)$ consists of diagonal and anti-diagonal matrices. Conjugation by a diagonal matrix leaves diagonal matrices fixed and for an anti-diagonal $h_1$ we have $\leftindex^{h_1}{\left(\begin{smallmatrix}1&0\\0&\lambda \end{smallmatrix}\right)}=\left(\begin{smallmatrix}\lambda&0\\0&1 \end{smallmatrix}\right), \leftindex^{h_1}{\left( \begin{smallmatrix}\lambda&0\\0&\lambda \end{smallmatrix}\right)}=\left( \begin{smallmatrix}\lambda&0\\0&\lambda \end{smallmatrix}\right)$. We obtain then $\leftindex^{h_1}{\tau}_{a,c}=\tau_{-a,a+c}$, for $h_1=\left( \begin{smallmatrix}0&1\\1&0 \end{smallmatrix}\right)$ and no further equivalences. This gives $(p(p-1)/2)+p-1=(p^2+p-2)/2$ braces. With $\lambda$ an element of order $p$ in $(\Z(q))^*$, and taking into account the isomorphism from $(B_2,\cdot)$ into $\Z/(p)\times \Z/(p)$ given in Section \ref{elem}, their multiplicative laws are given by

\begin{equation}\label{mulnn7}
\left( \left( \begin{smallmatrix} x_1\\ y_1 \end{smallmatrix} \right),\left( \begin{smallmatrix} z_1\\ t_1 \end{smallmatrix} \right)\right)\cdot\left( \left( \begin{smallmatrix} x_2\\ y_2 \end{smallmatrix} \right),\left( \begin{smallmatrix} z_2\\ t_2 \end{smallmatrix} \right)\right)=\left( \left( \begin{smallmatrix} x_1+\lambda^{(z_1-t_1(t_1-1)/2)(a+c)+t_1}x_2\\ y_1+\lambda^{(z_1-t_1(t_1-1)/2)c+t_1}y_2 \end{smallmatrix} \right),\left( \begin{smallmatrix} z_1+z_2+t_1t_2\\ t_1+t_2 \end{smallmatrix} \right)\right),
\end{equation}

\noindent
for some $(a,c) \in (\Z/(p))^*\times \Z/(p)$ where the braces corresponding to $(a,c)$ and $(-a,a+c)$ are isomorphic, and

\begin{equation}\label{mulnn8}
\left( \left( \begin{smallmatrix} x_1\\ y_1 \end{smallmatrix} \right),\left( \begin{smallmatrix} z_1\\ t_1 \end{smallmatrix} \right)\right)\cdot\left( \left( \begin{smallmatrix} x_2\\ y_2 \end{smallmatrix} \right),\left( \begin{smallmatrix} z_2\\ t_2 \end{smallmatrix} \right)\right)=\left( \left( \begin{smallmatrix} x_1+\lambda^{z_1-t_1(t_1-1)/2}x_2\\ y_1+\lambda^{z_1-t_1(t_1-1)/2+ct_1}y_2 \end{smallmatrix} \right),\left( \begin{smallmatrix} z_1+z_2+t_1t_2\\ t_1+t_2 \end{smallmatrix} \right)\right),
\end{equation}

\noindent
for some $c \in (\Z/(p))^*$. Besides, we have the direct product of $B_1$ and $B_2$ with multiplicative law given by

\begin{equation}\label{mulnn9}
\left( \left( \begin{smallmatrix} x_1\\ y_1 \end{smallmatrix} \right),\left( \begin{smallmatrix} z_1\\ t_1 \end{smallmatrix} \right)\right)\cdot\left( \left( \begin{smallmatrix} x_2\\ y_2 \end{smallmatrix} \right),\left( \begin{smallmatrix} z_2\\ t_2 \end{smallmatrix} \right)\right)=\left( \left( \begin{smallmatrix} x_1+x_2\\ y_1+y_2 \end{smallmatrix} \right),\left( \begin{smallmatrix} z_1+z_2+t_1t_2\\ t_1+t_2 \end{smallmatrix} \right)\right),
\end{equation}

\subsubsection{$B_1$ nontrivial brace}

In this case, $\Aut B_1=\left\{ \left( \begin{smallmatrix} d^2 & b \\ 0&d \end{smallmatrix} \right) \, : \, d \in (\Z/(q))^*, b \in \Z/(q) \right\} \subset \GL(2,q)$. Since the only subgroup of order $p$ of $\Aut B_1$ is $\left\langle \left( \begin{smallmatrix} \lambda^2 & 0 \\ 0&\lambda \end{smallmatrix} \right)\right\rangle$, for $\lambda \in (\Z/(q))^*$ of order $p$, a nontrivial morphism $\tau$ from $\Z/(p)\times \Z/(p)$ to $\Aut B_1$  has an order $p$ kernel and is defined by $\tau(u)=\left( \begin{smallmatrix} \lambda^2 & 0 \\ 0&\lambda \end{smallmatrix} \right)$, $\tau(v)=\Id$, for some $\Z/(p)$-basis $(u,v)$ of $\Z/(p)\times \Z/(p)$.

\noindent
{\bf If $B_2$ is trivial,}  $\Aut B_2=\GL(2,p)$. For $\tau$ defined by $\tau(u)=\left( \begin{smallmatrix} \lambda^2 & 0 \\ 0&\lambda \end{smallmatrix} \right)$, $\tau(v)=\Id$, for some basis $(u,v)$ of $(B_2,\cdot)=\Z/(p)\times \Z/(p)$, we have $\tau=\tau_0 h_2$, for $h_2$ defined by $h_2(u)=\left( \begin{smallmatrix} 1 \\ 0 \end{smallmatrix} \right), h_2(v)=\left( \begin{smallmatrix}  0 \\ 1\end{smallmatrix} \right)$ and $\tau_0$ defined by $\tau_0\left( \begin{smallmatrix} 1 \\ 0 \end{smallmatrix} \right)=\left( \begin{smallmatrix} \lambda^2 & 0 \\ 0&\lambda \end{smallmatrix} \right),\tau_0\left( \begin{smallmatrix} 0 \\ 1 \end{smallmatrix} \right)=\Id$. We obtain then one brace, whose multiplicative law is given by

\begin{equation}\label{mulnn10}
\left( \left( \begin{smallmatrix} x_1\\ y_1 \end{smallmatrix} \right),\left( \begin{smallmatrix} z_1\\ t_1 \end{smallmatrix} \right)\right)\cdot\left( \left( \begin{smallmatrix} x_2\\ y_2 \end{smallmatrix} \right),\left( \begin{smallmatrix} z_2\\ t_2 \end{smallmatrix} \right)\right)=\left(\left( \begin{smallmatrix} x_1+\lambda^{2z_1} x_2+\lambda^{z_1} y_1y_2\\ y_1 +\lambda^{z_1} y_2\end{smallmatrix} \right),\left( \begin{smallmatrix} z_1+z_2\\ t_1+t_2 \end{smallmatrix} \right)\right),
\end{equation}

\noindent
for $\lambda$ a fixed element of order $p$ in $(\Z/(q))^*$. Besides, we have the direct product whose multiplicative law is given by

\begin{equation}\label{mulnn11}
\left( \left( \begin{smallmatrix} x_1\\ y_1 \end{smallmatrix} \right),\left( \begin{smallmatrix} z_1\\ t_1 \end{smallmatrix} \right)\right)\cdot\left( \left( \begin{smallmatrix} x_2\\ y_2 \end{smallmatrix} \right),\left( \begin{smallmatrix} z_2\\ t_2 \end{smallmatrix} \right)\right)=\left(\left( \begin{smallmatrix} x_1+x_2+y_1y_2\\ y_1 +y_2\end{smallmatrix} \right),\left( \begin{smallmatrix} z_1+z_2\\ t_1+t_2 \end{smallmatrix} \right)\right).
\end{equation}

\noindent
{\bf If $B_2$ is nontrivial,} $\Aut B_2=\left\{ \left( \begin{smallmatrix} d^2 & b \\ 0&d \end{smallmatrix} \right) \, : \, d \in (\Z/(p))^*, b \in \Z/(p) \right\} \subset \GL(2,p)$. As in Section \ref{three}, we obtain  that the set of nontrivial morphisms $\tau$ from $\Z/(p)\times \Z/(p)$ to $\Aut B_1$ is precisely $\{ \tau_0 \left( \begin{smallmatrix} a & b\\ 0&1\end{smallmatrix} \right)\, : \, a \in (\Z/(p))^*, b \in \Z/(p) \}\cup \{ \tau_0 \left( \begin{smallmatrix} 0 & b\\ 1&0\end{smallmatrix} \right)\, : \, b \in (\Z/(p))^*\}$, for $\tau_0$ defined by $\tau_0\left( \begin{smallmatrix} 1 \\ 0 \end{smallmatrix} \right)=\left( \begin{smallmatrix} \lambda^2 & 0 \\ 0&\lambda \end{smallmatrix} \right),\tau_0\left( \begin{smallmatrix} 0 \\ 1 \end{smallmatrix} \right)=\Id$. Again, under the relation in Proposition \ref{prop}, we have three orbits corresponding to the matrices $\Id, \left( \begin{smallmatrix} a & 0\\ 0&1\end{smallmatrix} \right)$, for $a$ non-square, and $\left( \begin{smallmatrix} 0 & 1\\ 1&0\end{smallmatrix} \right)$. We obtain then three braces. Taking into account the isomorphism from $(B_2,\cdot)$ into $\Z/(p)\times \Z/(p)$ given in Section \ref{elem}, their multiplicative laws are given by

\begin{equation}\label{mulnn12}
\left( \left( \begin{smallmatrix} x_1\\ y_1 \end{smallmatrix} \right),\left( \begin{smallmatrix} z_1\\ t_1 \end{smallmatrix} \right)\right)\cdot\left( \left( \begin{smallmatrix} x_2\\ y_2 \end{smallmatrix} \right),\left( \begin{smallmatrix} z_2\\ t_2 \end{smallmatrix} \right)\right)=\left(\left( \begin{smallmatrix} x_1+\lambda^{2(z_1-t_1(t_1-1)/2)} x_2+\lambda^{z_1-t_1(t_1-1)/2} y_1y_2\\ y_1 +\lambda^{z_1-t_1(t_1-1)/2} y_2\end{smallmatrix} \right),\left( \begin{smallmatrix} z_1+z_2+t_1t_2\\ t_1+t_2 \end{smallmatrix} \right)\right),
\end{equation}

\begin{equation}\label{mulnn13}
\left( \left( \begin{smallmatrix} x_1\\ y_1 \end{smallmatrix} \right),\left( \begin{smallmatrix} z_1\\ t_1 \end{smallmatrix} \right)\right)\cdot\left( \left( \begin{smallmatrix} x_2\\ y_2 \end{smallmatrix} \right),\left( \begin{smallmatrix} z_2\\ t_2 \end{smallmatrix} \right)\right)=\left(\left( \begin{smallmatrix} x_1+(a\lambda^{2})^{(z_1-t_1(t_1-1)/2)} x_2+\lambda^{z_1-t_1(t_1-1)/2} y_1y_2\\ y_1 +\lambda^{z_1-t_1(t_1-1)/2} y_2\end{smallmatrix} \right),\left( \begin{smallmatrix} z_1+z_2+t_1t_2\\ t_1+t_2 \end{smallmatrix} \right)\right),
\end{equation}

\noindent
for $a$ is a fixed quadratic nonresidue modulo $p$, and

\begin{equation}\label{mulnn14}
\left( \left( \begin{smallmatrix} x_1\\ y_1 \end{smallmatrix} \right),\left( \begin{smallmatrix} z_1\\ t_1 \end{smallmatrix} \right)\right)\cdot\left( \left( \begin{smallmatrix} x_2\\ y_2 \end{smallmatrix} \right),\left( \begin{smallmatrix} z_2\\ t_2 \end{smallmatrix} \right)\right)=\left(\left( \begin{smallmatrix} x_1+\lambda^{(z_1-t_1(t_1-1)/2)} x_2+\lambda^{2(z_1-t_1(t_1-1)/2)} y_1y_2\\ y_1 +\lambda^{2(z_1-t_1(t_1-1)/2)} y_2\end{smallmatrix} \right),\left( \begin{smallmatrix} z_1+z_2+t_1t_2\\ t_1+t_2 \end{smallmatrix} \right)\right).
\end{equation}

\noindent
Besides, we have the direct product with multiplicative law defined by

\begin{equation}\label{mulnn15}
\left( \left( \begin{smallmatrix} x_1\\ y_1 \end{smallmatrix} \right),\left( \begin{smallmatrix} z_1\\ t_1 \end{smallmatrix} \right)\right)\cdot\left( \left( \begin{smallmatrix} x_2\\ y_2 \end{smallmatrix} \right),\left( \begin{smallmatrix} z_2\\ t_2 \end{smallmatrix} \right)\right)=\left(\left( \begin{smallmatrix} x_1+ x_2+y_1y_2\\ y_1 +y_2\end{smallmatrix} \right),\left( \begin{smallmatrix} z_1+z_2+t_1t_2\\ t_1+t_2 \end{smallmatrix} \right)\right).
\end{equation}

\vspace{0.3cm}
Summing up, we have obtained the following result.

\begin{theorem} Let $p$ and $q$ be primes satisfying $q>p, q\geq 5, p \mid q-1, p\nmid q+1$ and $p^2 \nmid q-1$. There are $\dfrac{p^2+5p}{2}+14$ (resp. $\dfrac{p^2+5p}{2}+13$) braces with additive group $\Z/(pq)\times \Z/(pq)$ if $p\equiv 1 \pmod{4}$ (resp. if $p\equiv 3 \pmod{4}$).

\begin{enumerate}[a)]
\item There are four of them with multiplicative group $\Z/(pq)\times \Z/(pq)$;
\item for each of the matrices $M$ in \eqref{mat} different from $\left( \begin{smallmatrix} \lambda & 0 \\ 0& \lambda^{-1} \end{smallmatrix} \right)$ and $\left( \begin{smallmatrix} \lambda&0\\ 0&\lambda^{(p+1)/2} \end{smallmatrix} \right)$, there are four of them with multiplicative group $(\Z/(q)\times \Z/(q))\rtimes_M (\Z/(p)\times \Z/(p))$;
\item for $M=\left( \begin{smallmatrix} \lambda & 0 \\ 0& \lambda^{-1} \end{smallmatrix} \right)$, there are four (resp. three) such braces with multiplicative group $(\Z/(q)\times \Z/(q))\rtimes_M (\Z/(p)\times \Z/(p))$,  if $p\equiv 1 \pmod{4}$ (resp. if $p\equiv 3 \pmod{4}$);
\item for $M=\left( \begin{smallmatrix} \lambda&0\\ 0&\lambda^{(p+1)/2} \end{smallmatrix} \right)$, there are eight of them with multiplicative group $(\Z/(q)\times \Z/(q))\rtimes_M (\Z/(p)\times \Z/(p))$;
\item there are $(p^2+p)/2$ of them with  multiplicative group $(\Z/(q)\times \Z/(q))\rtimes_{\lambda} (\Z/(p)\times \Z/(p))$.
\end{enumerate}
\end{theorem}

\section*{Acknowledgments}

This work was supported by grant PID2019-107297GB-I00, Ministerio de Ciencia, Innovación y Universidades.

I am very grateful to the referee for indications and corrections which helped to improve substantially this manuscript.

\end{document}